\documentclass[11pt]{article}

\usepackage{amsmath}
\usepackage{graphicx}
\usepackage{subcaption}
\usepackage[top=30mm, bottom=30mm, left=27mm, right=27mm]{geometry}
\usepackage{amsfonts}
\usepackage{mathtools}
\usepackage{subcaption}
\usepackage{amssymb}
\usepackage{amsthm}
\usepackage{chngpage}
\graphicspath{ {Images/} }
\usepackage{hyperref}
\hypersetup{
	colorlinks=true,
	linkcolor=black,
	citecolor=black,
	filecolor=black,      
	urlcolor=black,
}
\usepackage{enumitem}
\usepackage{makeidx}
\usepackage{bm}
\makeindex

\makeatletter
\renewenvironment{proof}[1][\proofname] {\par\pushQED{\qed}\normalfont\topsep6\p@\@plus6\p@\relax\trivlist\item[\hskip\labelsep\bfseries#1\@addpunct{.}]\ignorespaces}{\popQED\endtrivlist\@endpefalse}
\makeatother

\newcommand{\HH}{\mathcal{H}}
\newcommand{\xx}{\textbf{x}}
\newcommand{\yy}{\textbf{y}}
\newcommand{\RR}{\mathbb{R}}
\newcommand{\diam}{\operatorname{diam}}

\newtheorem{proposition}{Proposition}

\newtheorem{theorem}[proposition]{Theorem}

\usepackage{amsthm}

\theoremstyle{definition}

\title{A note on antichains in the continuous cube}
\author{Barnab\'{a}s Janzer\thanks{Department of Pure Mathematics and Mathematical Statistics, Centre for Mathematical Sciences, University of Cambridge, Wilberforce Road, Cambridge CB3 0WB, United Kingdom. Email: bkj21@cam.ac.uk}}
\date{\vspace{-21pt}}

\begin{document}
	\maketitle
	
\begin{abstract}
It is well-known that an antichain in the poset $[0,1]^n$ must have
measure zero. Engel, Mitsis, Pelekis and Reiher showed that in fact it must have $(n-1)$-dimensional Hausdorff measure at most $n$, and they conjectured that this bound can be attained.
In this note we show that, for every $n$, such an antichain does indeed exist.
\end{abstract}\vspace{-5pt}	

\section{Introduction}

We can partially order $[0,1]^n$ by writing $\xx\leq \yy$ if $x_i\leq y_i$ for each $i=1,\dots,n$. A subset $A\subseteq [0,1]^n$ is called an \textit{antichain} if it contains no two distinct elements $\xx,\yy$ with $\xx\leq \yy$. It is natural to ask for continuous generalisations of Sperner's theorem 
(which describes the largest possible size of an antichain in the discrete cube $\{0,1\}^n$, see e.g. \cite{bollobas1986combinatorics} for background). It is well-known that any antichain in $[0,1]^n$ must have Lebesgue measure zero. However, as shown by Engel, Mitsis, Pelekis and Reiher \cite{engel2018projection}, we can make interesting statements about the sizes of the antichains if we consider their Hausdorff measure. The definition and some properties of the Hausdorff measure are recalled in the Appendix, following \cite{engel2018projection}, to make the paper self-contained.

	Engel, Mitsis, Pelekis and Reiher \cite{engel2018projection} proved the following result.
	\begin{theorem}[Engel, Mitsis, Pelekis and Reiher \cite{engel2018projection}] \label{Thm_EMPR}
	Any antichain in $[0,1]^n$ has Hausdorff dimension at most $n-1$ and $(n-1)$-dimensional Hausdorff measure at most $n$. Moreover, these bounds cannot be improved (that is, there are antichains with $(n-1)$-dimensional Hausdorff measure arbitrarily close to $n$).
	\end{theorem} 
They also conjectured that there are antichains in $[0,1]^n$ of $(n-1)$-dimensional Hausdorff measure exactly $n$. In this note, we construct such an antichain for each $n$.
\begin{theorem}\label{Thm_antichainconstruction}
	For every $n$ there is an antichain in $[0,1]^n$ of $(n-1)$-dimensional Hausdorff measure $n$.
\end{theorem}
	
\section{The construction}
	\begin{proof}[Proof of Theorem \ref{Thm_antichainconstruction}]
Let $f$ be a strictly increasing singular function from $[0,1]$ onto $[0,1]$, that is, a strictly increasing continuous surjective function $[0,1]\to [0,1]$ for which there is a set $S\subseteq (0,1)$ of measure $1$ such that $f$ is differentiable with derivative zero at all points of $S$. (See e.g. \cite{salem1943some} for the construction of such a function.) As noted by Engel, Mitsis, Pelekis and Reiher \cite{engel2018projection}, the graph of $1-f$ has Hausdorff measure $2$ (see \cite{foran1999length}), giving an example of an antichain with the required properties when $n=2$. We will generalise this construction for any $n\geq 2$. (The case $n=1$ is trivial.)\medskip
	
	First, we construct a function $p: (0,1)^{n-1}\to (0,1)$ with the following properties.
	\begin{enumerate}
		\item The function $p$ is differentiable at every point in $(0,1)^{n-1}$ except maybe when two of the coordinates are equal.
		
		\item Whenever $n-2$ coordinates are fixed, then the resulting function $(0,1)\to (0,1)$ is strictly increasing and surjective. In particular, if $\xx<\yy$ for some $\xx, \yy\in(0,1)^{n-1}$, then $p(\xx)<p(\yy)$. (We write $\xx<\yy$ if $\xx\leq\yy$ and $\xx\not=\yy$.)
	\end{enumerate}

	If $n=2$ then $p(x)=x$ works. For $n\geq 3$, we define $p$ as follows. Given $(x_1,\dots,x_{n-1})\in (0,1)^{n-1}$, let $\sigma$ be a permutation of $\{1,\dots,n-1\}$ such that $x_{\sigma(1)}\leq\dots\leq x_{\sigma(n-1)}$, and set
	$$p(x_1,\dots,x_{n-1})=\frac{\prod_{i=1}^{n-2}x_{\sigma(i)}}{1-x_{\sigma(n-1)}+\prod_{i=1}^{n-2}x_{\sigma(i)}}.$$
	It is clear that this satisfies the conditions above.
	
	Define $F: (0,1)^{n-1}\to (0,1)$ by
	$$F(x_1,\dots, x_{n-1})=1-p(f(x_1),f(x_2),\dots,f(x_{n-1})),$$
	and write $A$ for the graph of $F$. Note that, since $f$ is strictly increasing, Property 2 above shows that if $\xx<\yy$ then $F(\xx)>F(\yy)$ and so $A$ is an antichain. We show that $A$ has $(n-1)$-dimensional Hausdorff measure at least $n$. (The measure is then exactly $n$ by Theorem \ref{Thm_EMPR}.)
	
	Write $\pi_i:\RR^n\to\RR^{n-1}$ for the projection $\pi(x_1,\dots,x_n)=(x_1,\dots,x_{i-1},x_{i+1},\dots,x_n)$. We find $n$ disjoint subsets $A_1, \dots, A_n$ of $A$ such that $\pi_i(A_i)$ has measure $1$ for each $i$. This is sufficient, since the projections $\pi_i$ cannot increase the Hausdorff measure, so each $A_i$ must then have measure at least $1$.
	
	Recall that there is a subset $S\subseteq (0,1)$ of measure $1$ such that $f$ is differentiable with derivative zero at each point of $S$. We set
	$$B_n=S^{n-1}$$
	and, for $i=1, \dots, n-1$,
	$$B_i=\{\xx\in (0,1)^{n-1}: x_j\in S\textnormal{ for $j\not=i$ and }x_i\not\in S\}.$$

	Let $A_i$ be the piece of the graph of $F$ corresponding to domain $B_i$. Note that the $A_i$ are disjoint. It is clear that $\pi_n(A_n)=B_n$ has measure $1$. To show that $\pi_i(A_i)$ has measure $1$ for $i=1,\dots, n-1$, it suffices to consider the case $i=n-1$.
	
	Let $\xx=(x_1,\dots,x_{n-2})\in S^{n-2}$ be a point with all coordinates distinct, we show that the set $\{x_n\in (0,1):(x_1,\dots,x_{n-2},x_n)\in\pi_{n-1}(A_{n-1})\}$ has measure 1. (Since the set of such points $\xx\in S^{n-2}$ has measure 1, this then gives the result.) Write $T=\{x_1,\dots,x_{n-2}\}$. Consider the function $g_\xx: (0,1)\to (0,1)$ given by $g_\xx(t)=F(x_1,\dots,x_{n-2},t)$. Note that $g_\xx$ is strictly decreasing and surjective by the properties of $p$ and $f$. Moreover, whenever $t\in S\setminus{T}$, then the chain rule gives that $g_\xx$ is differentiable at $t$ with derivative zero. It follows that $g_\xx(S\setminus T)$ has measure $0$, and hence (using that $T$ is finite and $g_\xx$ is surjective onto $(0,1)$) $g_\xx((0,1)\setminus{S})$ has measure 1. But this says exactly that $\{x_n\in (0,1):(x_1,\dots,x_{n-2},x_n)\in\pi_{n-1}(A_{n-1})\}$ has measure 1, giving the result.
\end{proof}

We remark that when $n=3$, the function $p: (0,1)^2\to (0,1)$ has a nice geometric interpretation: $p(\xx)$ is given by the (first) coordinate of the projection of $\xx$ onto the diagonal $\{(t,t):t\in(0,1)\}$ from the point $(1,0)$ if $\xx$ is below the diagonal and from $(0,1)$ if it is above.
	
\bibliography{Bibliography}
\bibliographystyle{abbrv}	
	
\appendix
\section*{Appendix: The Hausdorff measure}
Following \cite{engel2018projection}, we recall the definition and some properties of the Hausdorff measure.

For a non-empty subset $U$ of $\RR^n$, let $\diam{U}$ denote its diameter. For any real number $s\geq 0$, write $\alpha_s=\frac{\pi^{s/2}}{2^s \Gamma(s/2+1)}$ for the volume of the $s$-dimensional sphere of radius $1/2$. For $s\geq 0$, $\delta>0$ and $A\subseteq \RR^n$, write
$$\HH_\delta^s(A)=\alpha_s \inf\left\{\sum_{i=1}^{\infty}{\diam(U_i)^s}: A\subseteq \bigcup_{i=1}^{\infty}U_i\textnormal{ and $\diam(U_i)\leq \delta$ for each $i$}\right\}.$$
Note that as $\delta$ decreases, the value of $\HH_\delta^s(A)$ increases, so we may set 
$$\HH^s(A)=\lim_{\delta\to 0^{+}}{\HH_\delta^s(A)}.$$
It can be shown that $\HH^s$ restricts to a measure, called the \textit{$s$-dimensional Hausdorff measure}, on a $\sigma$-algebra containing all the Borel measurable sets. Note that the scaling by $\alpha_s$ guarantees that the Hausdorff measure and the Lebesgue measure agree when $s=n$. Furthermore, it can be shown that
$$\dim_H{(A)}=\inf\{s\geq 0: \HH^{s}(A)=0\}$$
has the property that if $s<\dim_H(A)$ then $\HH^s(A)=0$ and if $s>\dim_H(A)$ then $\HH^s(A)=\infty$, so the only 'interesting' value of $\HH^s(A)$ occurs when $s=\dim_H(A)$. The value $\dim_H(A)$ is called the \textit{Hausdorff dimension} of $A$. See  \cite{engel2018projection} for references on the topic.
	
\end{document}